\documentclass[a4paper,11pt]{article}
\usepackage{graphicx}
\usepackage{amsfonts}
\usepackage{amsmath}
\usepackage{amssymb}
\usepackage{amsthm}
\usepackage{hyperref}

\theoremstyle{plain}
\newtheorem{lemma}{Lemma}[section]
\newtheorem{theorem}[lemma]{Theorem}
\newtheorem{corollary}[lemma]{Corollary}

\theoremstyle{definition}
\newtheorem{definition}{Definition}[section]

\theoremstyle{remark}
\newtheorem{remark}{Remark}[section]

\begin{document}

%
%
\title{New upper bounds for the period of a negative orientable sequence}
\author{Chris J. Mitchell and Peter R.
Wild\\Royal Holloway, University of London, Egham TW20 0EX, UK}




\maketitle

\begin{abstract}
Negative orientable sequences, i.e.\ periodic sequences with elements from a finite alphabet of
size at least three in which an $n$-tuple or the negative of its reverse appears at most once
in a period of the sequence, were introduced by Alhakim et al.\ in 2024. The main goal in
defining them was as a means of generating orientable sequences, which have automatic position
location applications, although they are potentially of interest in their own right.  In this
paper we develop new upper bounds on the period of negative orientable sequences which, for
$n>2$, are significantly sharper than the previous known bound.  The approach used to develop
the new bounds involves examining the nodes in the subgraph of the de Bruijn graph
corresponding to a negative orientable sequence, and to consider the implications of the fact
that the in-degree of every vertex in this subgraph must equal the out-degree. However, despite
improving the bounds, a gap remains between the largest known period for a negative orientable
sequence and the corresponding bounds for every $n>2$.

\end{abstract}

\section{Introduction}

Orientable sequences are periodic sequences with elements from an alphabet of size $k$ with the
property that, for some $n$ (the \emph{order}), any $n$-tuple occurs at most once in a period of
the sequence in either direction.  Clearly an orientable sequence of order $n$, an
$\mathcal{OS}_k(n)$, is also an $\mathcal{OS}_k(m)$ for every $m\geq n$. Orientable sequences were
introduced in the early 1990s \cite{Burns92,Burns93} in connection with position location
applications.

The binary case was studied by Dai et al. \cite{Dai93}, who presented an upper bound and a
construction method yielding sequences with asymptotically optimal period.  Further methods of
construction for this case appeared much more recently \cite{Gabric24,Gabric24b,Mitchell22}. The
general alphabet case (i.e.\ for $k\geq2$) was first considered by Alhakim et al.\
\cite{Alhakim24a}. Subsequently, construction methods for this case have been described,
\cite{Gabric25,Mitchell25a,Mitchell26}, and upper bounds for the period of a $k$-ary orientable
sequence were given in \cite{Mitchell25a} and improved in \cite{Mitchell26}.  Although the
construction method of Gabri\'{c} and Sawada \cite{Gabric25} has been shown to yield sequences of
asymptotically optimal period, the problem of determining the precise value of the largest period
for an $\mathcal{OS}_k(n)$ remains an open question except for small values of $n$ --- for further
details of the current state of the art see \cite{Mitchell26}.

Alhakim at al.\ \cite{Alhakim24a} introduced the notion of a \emph{negative orientable sequence},
the main focus of this paper, to help construct new orientable sequences. A negative orientable
sequence of order $n$, $\mathcal{NOS}_k(n)$, is a periodic $k$-ary sequence in which an $n$-tuple
or the negative of its reverse appears at most once in a period of the sequence. Such sequences
were subsequently studied further in \cite{Mitchell25a,Mitchell26}, where a range of methods of
construction are described, and an upper bound on the period of a $\mathcal{NOS}_k(n)$ is given in
\cite{Mitchell26}.  Much like the orientable sequence case, determining the precise value of the
largest period for an $\mathcal{NOS}_k(n)$ remains unresolved, in this case even for small $n$.

In this paper we use the same approach as employed in \cite{Mitchell26} to develop new period upper
bounds for orientable sequences to develop new and sharper bounds for the period of a negative
orientable sequence.  The main tool is to examine nodes in the subgraph of the de Bruijn graph
corresponding to a negative orientable sequence, and to consider the implications of the fact that
the in-degree of every vertex in this subgraph must equal the out-degree.

The remainder of the paper is structured as follows.  Section~\ref{section:preliminaries} gives
fundamental definitions, some elementary counting results for $n$-tuples, and a brief introduction
to the de Bruijn graph. This is followed in Section~\ref{section:bounds} by the development of the
new period bounds. Finally, Section~\ref{section:conclusions} concludes the paper, summarising the
state of knowledge regarding the maximum period for negative orientable sequences.

\section{Preliminaries} \label{section:preliminaries}

\subsection{Basic Definitions}

We need the following key definitions, largely following the notational conventions of
\cite{Mitchell25a}.  We refer throughout to a $k$-ary $n$-tuple to mean a sequence of length $n$ of
symbols drawn from $\mathbb{Z}_k$.  Note that we assume that $k>2$ throughout since the negative of
a tuple equals the tuple when $k=2$.

Since we are interested in tuples occurring either forwards or backwards in a sequence, we
introduce the notion of a reversed tuple, so that if $\mathbf{u} = (u_0,u_1,\ldots,u_{n-1})$ is
a $k$-ary $n$-tuple then $\mathbf{u}^R = (u_{n-1},u_{n-2}, \ldots,u_0)$ is its \emph{reverse}.
We are also interested in negating all the elements of a tuple, and hence if $\mathbf{u} =
(u_0,u_1,\ldots,u_{n-1})$ is a $k$-ary $n$-tuple, we write $-\mathbf{u}$ for
$(-u_0,-u_1,\ldots,-u_{n-1})$. We write $\mathbf{s}_n(i)$ for the $n$-tuple occurring at
position $i$ in a sequence $S=(s_i)$.

\begin{definition}[\cite{Alhakim24a}]
A $k$-ary \emph{$n$-window sequence $S = (s_i)$} is a periodic sequence of elements from
$\mathbb{Z}_k$ ($k>1$, $n>1$) with the property that no $n$-tuple appears more than once in a
period of the sequence, i.e.\ with the property that if $\mathbf{s}_n(i) = \mathbf{s}_n(j)$ for
some $i,j$, then $i \equiv j \pmod m$ where $m$ is the period of the sequence.
\end{definition}

This paper is concerned with a special class of $n$-window sequences: negative orientable
sequences.

\begin{definition}[\cite{Alhakim24a}]  \label{definition:NOS}
A $k$-ary $n$-window sequence $S=(s_i)$ is said to be a \emph{negative orientable sequence of order
$n$} (a $\mathcal{NOS}_k(n)$) if $\mathbf{s}_n(i)\not=-{\mathbf{s}_n(j)}^R$, for any $i,j$.
\end{definition}

\begin{definition}[\cite{Mitchell25a}]
Suppose $n\geq 1$ and $k>2$.  A $k$-ary $n$-tuple $\mathbf{u}=(u_0,u_1,\dots,u_{n-1})$ is said
to be \emph{negasymmetric} if $u_i=-u_{n-1-i}$ for every $i$, $0\leq i\leq n-1$, i.e.\ if
$\mathbf{u}=-\mathbf{u}^R$.
\end{definition}

Clearly an $\mathcal{NOS}_k(n)$ cannot contain any negasymmetric $n$-tuples.  This observation
yields a very simple bound, but \cite{Mitchell25a} gives a sharper bound, as follows.

\begin{theorem}[Theorem 3.5 of \cite{Mitchell25a}] \label{result:NOS-bounds}
Suppose that $S=(s_i)$ is a $\mathcal{NOS}_k(n)$ ($k\geq2$, $n\geq 2$). Then the period of $S$ is
at most
\begin{align*}
(k^n - k^{\lfloor n/2\rfloor} - k^{\lfloor (n-1)/2\rfloor} + 1)/2 & ~~\mbox{if $k$ is odd}, \\
(k^n - 2k^{(n-1)/2})/2 - 1 & ~~\mbox{if $k$ is even and $n$ is odd}, \\
(k^n - k^{n/2})/2 - 1 & ~~\mbox{if $k$ is even and $n$ is even}.
\end{align*}
\end{theorem}

The main goal of this paper is to give a significantly sharper bound.  A comparison of the new and
old bounds is provided in Table~\ref{table:new_NOS_periods_bounds}.

\begin{definition}
Suppose $n\geq2$ and $k\geq2$.  If $\mathbf{a}=(a_0,a_1,\ldots,a_{n-1})$ is a $k$-ary $n$-tuple,
then $\mathbf{a}$ is said to be \emph{uniform} if and only if $a_0=a_1=\dots=a_{n-1}$.
\end{definition}

\begin{definition}
Suppose $n\geq2$ and $k\geq2$.  A $k$-ary $n$-tuple $\mathbf{u}=(u_0,u_1,\ldots,u_{n-1})$ is said
to be \emph{alternating} if and only if there exist $c_0$ and $c_1$ ($c_0\neq c_1$) such that
$u_{2i}=c_0$ and $u_{2i+1}=c_1$ for every $i$.
\end{definition}

\begin{remark}
Since the above definition requires $c_0\neq c_1$, an alternating $n$-tuple is always non-uniform.
\end{remark}

\begin{definition}
Suppose $n\geq2$ and $k>2$.  A $k$-ary $n$-tuple $\mathbf{u}=(u_0,u_1,\ldots,u_{n-1})$ is said to
be \emph{negatively-alternating} if and only if $u_{i+1}=-u_i$ for every $i$, $0\leq i\leq n-2$.
\end{definition}

\begin{remark}
A negatively-alternating $n$-tuple is uniform if $u_0=0$ or $u_0=k/2$ ($k$ even).
\end{remark}

\begin{definition}
Suppose $n\geq2$ and $k>2$.  An $n$-tuple $(a_0,a_1,\ldots,a_{n-1})$ is said to be
\emph{left-semi-negasymmetric} (or \emph{left-sns} for short) if $a_i=-a_{n-i-2}$, $0\leq i\leq
n-2$. Equivalently, $(a_0,a_1,\ldots,a_{n-1})$ is left-sns if and only if
$(a_0,a_1,\ldots,a_{n-2})$ is negasymmetric.

Analogously, an $n$-tuple $(a_0,a_1,\ldots,a_{n-1})$ is said to be \emph{right-semi-negasymmetric}
(or \emph{right-sns} for short) if $a_i=-a_{n-i}$, $1\leq i\leq n-1$.  Equivalently,
$(a_0,a_1,\ldots,a_{n-1})$ is right-sns if and only if $(a_1,a_2,\ldots,a_{n-1})$ is negasymmetric.
\end{definition}

\subsection{Counting Sets of $n$-tuples}

The following simple results will be of use below.

\begin{lemma}  \label{lemma:NOS_tuple_numbers}
Suppose $n\geq1$ and $k>2$.  Then:
\begin{itemize}
\item[i)]  the number of $k$-ary negasymmetric $n$-tuples is
\begin{align*}
k^{(n-1)/2} & \mbox{~~if $n$ is odd and $k$ is odd} \\
2k^{(n-1)/2} & \mbox{~~if $n$ is odd and $k$ is even;} \\
k^{n/2} & \mbox{~~if $n$ is even;}
\end{align*}
\item[ii)] the number of uniform $k$-ary $n$-tuples is $k$;
\item[iii)] the number of negatively-alternating $k$-ary $n$-tuples is $k$;
\item[iv)] the number of $n$-tuples that are both uniform and negatively-alternating is 1 if
    $k$ is odd and 2 if $k$ is even;
\item[v)] the number of uniform negasymmetric $k$-ary $n$-tuples is 1 if $k$ is odd and 2 if
    $k$ is even;
\item[vi)] the number of negatively-alternating negasymmetric $k$-ary $n$-tuples is 1 if $n$
    and $k$ are odd, 2 if $n$ is odd and $k$ is even, and $k$ if $n$ is even;
\item[vii)] the number of alternating negasymmetric $k$-ary $n$-tuples is zero if $n$ and $k$
    are both odd, 2 if $n$ is odd and $k$ is even, $k-1$ if $n$ is even and $k$ is odd, and
    $k-2$ if $n$ and $k$ are both even; moreover, if $n$ is even they are
    negatively-alternating, but they are not if $n$ is odd and $k$ is even;
\item[viii)] the number of left-sns $k$-ary $n$-tuples is:
\begin{align*}
k^{(n+1)/2} & \mbox{~~if $n$ is odd;} \\
k^{n/2} & \mbox{~~if $n$ is even and $k$ is odd;} \\
2k^{n/2} & \mbox{~~if $n$ and $k$ are both even;}
\end{align*}
\item[ix)] the number of non-uniform left-sns $k$-ary $n$-tuples is:
\begin{align*}
k^{(n+1)/2}-1 & \mbox{~~if $n$ and $k$ are both odd;} \\
k^{(n+1)/2}-2 & \mbox{~~if $n$ is odd and $k$ is even;} \\
k^{n/2}-1 & \mbox{~~if $n$ is even and $k$ is odd;} \\
2k^{n/2}-2 & \mbox{~~if $n$ and $k$ are both even;}
\end{align*}
\item[x)] the number of non-negatively-alternating left-sns $k$-ary $n$-tuples is:
\begin{align*}
k^{(n+1)/2}-k & \mbox{~~if $n$ is odd;} \\
k^{n/2}-1 & \mbox{~~if $n$ is even and $k$ is odd;} \\
2k^{n/2}-2 & \mbox{~~if $n$ and $k$ are both even.}
\end{align*}
\item[xi)] if $n>2$, the number of non-uniform non-alternating left-sns $k$-ary $n$-tuples is:
\begin{align*}
k^{(n+1)/2}-k & \mbox{~~if $n$ is odd;} \\
k^{n/2}-1 & \mbox{~~if $n$ is even and $k$ is odd;} \\
2k^{n/2}-4 & \mbox{~~if $n$ and $k$ are both even;}
\end{align*}
\end{itemize}
Finally observe that (viii), (ix), (x) and (xi) also hold if left-sns is replaced with right-sns.
\end{lemma}

\begin{proof}
(i) is Lemma 3.1 of \cite{Mitchell25a}.  (ii), (iii) and (iv) are immediate.  (v) follows by
observing that a uniform negasymmetric $n$-tuple must have every entry equal to zero or $k/2$.

For (vi), if $n$ is odd then every element in the $n$-tuple must be either 0 or $k/2$ and hence the
$n$-tuple must be uniform; the result follows from (v).  If $n$ is even then every
negatively-alternating $n$-tuple is negasymmetric, and the result follows from (iii).

For (vii), if $n$ is odd then every element in the $n$-tuple must be either 0 or $k/2$, and the
result follows from (v) since an alternating $n$-tuple cannot be uniform.  If $n$ is even then an
alternating negasymmetric $n$-tuple must be negatively-alternating, and the result follows from
(vi) and (iv) since an alternating $n$-tuple cannot be uniform.

(viii) follows from (i) and observing that an $n$-tuple is left-sns if and only if its first $n-1$
entries form a negasymmetric $(n-1)$-tuple.  A left-sns $n$-tuple can only be uniform if it is the
all-zero or the all-$k/2$ $n$-tuple, and (ix) follows from (v) and (viii).

For (x), If $n$ is odd, then a negatively-alternating $n$-tuple is left-sns and in this case the
number is simply (viii) minus (iii).  If $n$ is even, a negatively-alternating $n$-tuple is only
left-sns if it is also uniform, i.e.\ the number is (viii) minus (iv).

Finally, for (xi), first observe that an $n$-tuple cannot be both alternating and uniform by
definition, hence (xi) is simply (ix) less the number of alternating left-sns $n$-tuples.  Also,
the number of alternating left-sns $n$-tuples is the same as the number of alternating
negasymmetric $(n-1)$-tuples, i.e., from Lemma~\ref{lemma:NOS_tuple_numbers}(vii), zero if $n$ is
even and $k$ is odd, 2 if $n$ and $k$ are both even, $k-1$ if $n$ and $k$ are both odd, and $k-2$
if $n$ is odd and $k$ is even. \qed\end{proof}

\subsection{Negative Orientable Sequences and the de Bruijn Digraph}

We also need the following definitions.  Following Alhakim et al.\ \cite{Alhakim24a} and many
previous authors, we introduce the $k$-ary de Bruijn digraph.

\begin{definition}
For positive integers $n$ and $k$ greater than one, let $\mathbb{Z}_k^n$ be the set of all $k^n$
tuples of length $n$ with entries from the group $\mathbb{Z}_k$ of residues modulo $k$. The order
$n$ \emph{de Bruijn digraph}, $B_k(n)$, is a directed graph with $\mathbb{Z}^n_k$ as its vertex set
in which, for any two vertices $\textbf{x} = (x_0,x_1,\ldots,x_{n-1})$ and $\textbf{y} =
(y_0,y_1,\ldots,y_{n-1})$, the pair $(\textbf{x},\textbf{y})$ is an edge if and only if $y_i =
x_{i+1}$ for every $i$ ($0\leq i< n-1$). We label such an edge with the $(n+1)$-tuple
$(x_0,x_1,\ldots,x_{n-1},y_{n-1})$.
\end{definition}

Note that we have defined two ways of specifying an edge in $B_k(n)$, namely as either a pair of
vertices $(\textbf{a},\textbf{b})$, where $\textbf{a},\textbf{b}$ are $k$-ary $n$-tuples, or as a
single $k$-ary $(n+1)$-tuple $\textbf{x}$.  We refer to an edge in $B_k(n)$ and a $k$-ary
$(n+1)$-tuple interchangeably throughout.

\begin{definition}
Suppose $k>2$ and $n\geq2$.  Let $B_k^{-}(n-1)$ be the subgraph of the de Bruijn digraph
$B_k(n-1)$ with all the edges corresponding to negasymmetric $n$-tuples removed.
\end{definition}

The following elementary result follows immediately from Lemma~\ref{lemma:NOS_tuple_numbers}.

\begin{lemma}  \label{lemma:number_non-negasymmetric_tuples}
If $k>2$ and $n\geq2$ the number $N_k(n)$ of edges in $B_k^-(n-1)$ is
\begin{align*}
k^n-k^{(n-1)/2} & \mbox{~~if $n$ is odd and $k$ is odd} \\
k^n-2k^{(n-1)/2} & \mbox{~~if $n$ is odd and $k$ is even} \\
k^n-k^{n/2} & \mbox{~~if $n$ is even.}
\end{align*}
\end{lemma}

\begin{definition}
Suppose $k>2$ and $n\geq2$.  Suppose $S$ is an $\mathcal{NOS}_k(n)$.  Let $B^-(S,n)$ be the
subgraph of $B_k(n-1)$ with vertices those of $B_k(n-1)$ and with edges corresponding to those
$n$-tuples which appear in either $S$ or $-S^R$.  We refer to $B^-(S,n)$ as the
nega-sequence-subgraph.
\end{definition}

The following simple lemma is key.

\begin{lemma}  \label{lemma:B-Sn_basic_properties}
Suppose $k>2$ and $n\geq2$.  Suppose $S$ is an $\mathcal{NOS}_k(n)$ of period $m$.  Then:
\begin{itemize}
\item[i)] $B^-(S,n)$ contains $2m$ edges;
\item[ii)] every vertex of $B^-(S,n)$ has in-degree equal to its out-degree; and
\item[iii)] $B^-(S,n)$ is a subgraph of $B^-_k(n-1)$.
\end{itemize}
\end{lemma}

\begin{proof}
\begin{itemize}
\item[i)] Since $S$ has period $m$, a total of $2m$ $n$-tuples appear in $S$ and $-S^R$. They
    are all distinct since $S$ is an $\mathcal{NOS}_k(n)$.
\item[ii)] $S$ and $-S^R$ correspond to edge-disjoint Eulerian circuits in $B^-(S,n)$, and the
    result follows.
\item[iii)] This is immediate since an $\mathcal{NOS}_k(n)$ cannot contain any negasymmetric
    $n$-tuples.
\end{itemize}
\qed
\end{proof}

\section{The New Bounds}  \label{section:bounds}

\subsection{In-out-degree Constraints on the Nega-sequence-subgraph}
\label{subsection:NOS_in-out-constraints}

In this section and the next we consider properties of the nega-sequence-subgraph for a
$\mathcal{NOS}_k(n)$ $S$.  We first have the following.

\begin{lemma}  \label{lemma:B-_vertex_degree_k-1}
Suppose $k>2$ and $n\geq3$.  A vertex in $B^-_k(n-1)$ has:
\begin{itemize}
\item[i)] in-degree $k-1$ if and only if its label is left-sns; otherwise it has in-degree $k$;
\item[ii)] out-degree $k-1$ if and only if its label is right-sns; otherwise it has out-degree
    $k$.
\end{itemize}
\end{lemma}

\begin{proof}
For (i), the in-degree of every vertex in $B_k(n-1)$ is $k$.  However, if (and only if) an inbound
edge corresponds to a negasymmetric $n$-tuple, then this edge will not be in $B_k^-(n-1)$.  Such an
event can occur if and only if the vertex is labelled with a left-sns $(n-1)$-tuple, and there can
only be one such negasymmetric inbound edge.  The result follows.  The proof of (ii) follows using
an exactly analogous argument. \qed
\end{proof}

By Lemma~\ref{lemma:B-Sn_basic_properties}(ii), this immediately tells us that some edges in
$B^-_k(n-1)$ cannot occur in $B^-(S,n)$ if $S$ is a $\mathcal{NOS}_k(n)$.  However, before
describing exactly when this occurs, we first need the following simple result.

\begin{lemma}  \label{lemma:combining_lns-rns-symmetric}
Suppose $k>2$ and $n\geq4$.  Suppose the $(n-1)$-tuple $(a_0,a_1,\ldots,a_{n-2})$ is both
left-sns and right-sns.  Then

\begin{itemize}
\item[i)] if $n$ is odd and $k$ is odd then $(a_0,a_1,\ldots,a_{n-2})$ is the all-zero uniform
    $(n-1)$-tuple;
\item[ii)] if $n$ is odd and $k$ is even then $(a_0,a_1,\ldots,a_{n-2})$ is either uniform or
    alternating, where $a_i\in\{0,k/2\}$ for every $i$;
\item[iii)] if $n$ is even then $(a_0,a_1,\ldots,a_{n-2})$ is negatively-alternating.
\end{itemize}
Further, all the $(n-1)$-tuples in (i), (ii) and (iii) are both left-sns and right-sns.
\end{lemma}

\begin{proof}
Suppose the $(n-1)$-tuple $(a_0,a_1,\ldots,a_{n-2})$ is both left-sns and right-sns. Then
$a_i=-a_{n-i-3}$, $0\leq i\leq n-3$, and $a_i=-a_{n-i-1}$, $1\leq i\leq n-2$.  Since $n\geq4$ this
implies that there exist constants $c_0$ and $c_1$ such that $c_0=a_{2i}$, $0\leq 2i\leq n-2$, and
$c_1=a_{2j+1}$, $0\leq 2j+1\leq n-2$.

If $n$ is even then we have $a_{(n-2)/2}=-a_{(n-2)/2-1}$ (from left-semi-negasymmetry), and hence
$c_0=-c_1$, and (iii) follows.

If $n$ is odd then, from left-semi-negasymmetry, $a_{(n-3)/2}=-a_{(n-3)/2}$ and from
right-semi-negasymmetry $a_{(n-1)/2}=-a_{(n-1)/2}$, and hence $c_0=-c_0$ and $c_1=-c_1$.  If $k$ is
odd then we have $c_0=c_1=0$, and (i) follows.  if $k$ is even then $c_0$ and $c_1$ are both either
0 or $k/2$, and (ii) follows. \qed
\end{proof}

The following result follows immediately from Lemmas~\ref{lemma:B-_vertex_degree_k-1} and
\ref{lemma:combining_lns-rns-symmetric}.

\begin{corollary}  \label{corollary:nega-left-right-in-out}
Suppose $k>2$ and $n\geq4$. Suppose $S$ is an $\mathcal{NOS}_k(n)$ and consider a vertex in
$B_k^-(n-1)$ with label $\mathbf{a}=(a_0,a_1,\ldots,a_{n-2})$, where $\mathbf{a}$ is non-uniform.
\begin{itemize}
\item[i)] if $n$ and $k$ are both odd, and $\mathbf{a}$ is left-sns, then its in-degree is
    $k-1$ and its out-degree is $k$;
\item[ii)] if $n$ and $k$ are both odd and $\mathbf{a}$ is right-sns, then its out-degree is
    $k-1$ and its in-degree is $k$;
\item[iii)] if $n$ is odd and $k$ is even and $\mathbf{a}$ is left-sns and non-alternating,
    then its in-degree is $k-1$ and its out-degree is $k$;
\item[iv)] if $n$ is odd and $k$ is even and $\mathbf{a}$ is right-sns and non-alternating,
    then its in-degree is $k$ and its out-degree is $k-1$;
\item[v)] if $n$ is even and $\mathbf{a}$ is left-sns and not negatively-alternating, then its
    in-degree is $k-1$ and its out-degree is $k$;
\item[vi)] if $n$ is even and $\mathbf{a}$ is right-sns and not negatively-alternating, then
    its out-degree is $k-1$ and its in-degree is $k$.
\end{itemize}
\end{corollary}

The above corollary immediately tells us that certain edges in $B^-_k(n-1)$ cannot occur in
$B^-(S,n)$, as follows.

\begin{corollary}  \label{corollary:NOS_excluded_tuples_by_unequal_degree}
Suppose $k>2$ and $n\geq4$ and $S$ is an $\mathcal{NOS}_k(n)$. Then
\begin{itemize}
\item[i)] if $n$ and $k$ are both odd, for every vertex in $B^-_k(n-1)$ corresponding to a
    non-uniform left-sns $(n-1)$-tuple $(a_0,a_1,\ldots,a_{n-2})$, there is an edge
    $(a_0,a_1,\ldots,a_{n-2},x)$ in $B^-_k(n-1)$, for some $x$, that is not in $B^-(S,n)$.
\item[ii)] if $n$ and $k$ are both odd, for every vertex in $B^-_k(n-1)$ corresponding to a
    non-uniform right-sns $(n-1)$-tuple $(a_0,a_1,\ldots,a_{n-2})$, there is an edge
    $(y,a_0,a_1,\ldots,a_{n-2})$ in $B^-_k(n-1)$, for some $y$, that is not in $B^-(S,n)$.
\item[iii)] if $n$ is odd and $k$ is even, for every vertex in $B^-_k(n-1)$ corresponding to a
    non-uniform non-alternating left-sns $(n-1)$-tuple $(a_0,a_1,\ldots,a_{n-2})$, there is an
    edge $(a_0,a_1,\ldots,a_{n-2},x)$ in $B^-_k(n-1)$, for some $x$, that is not in $B^-(S,n)$.
\item[iv)] if $n$ is odd and $k$ is even, for every vertex in $B^-_k(n-1)$ corresponding to a
    non-uniform non-alternating right-sns $(n-1)$-tuple $(a_0,a_1,\ldots,a_{n-2})$, there is an
    edge $(y,a_0,a_1,\ldots,a_{n-2})$ in $B^-_k(n-1)$, for some $y$, that is not in $B^-(S,n)$.
\item[v)] if $n$ is even, for every vertex in $B^-_k(n-1)$ corresponding to a
    non-negatively-alternating left-sns $(n-1)$-tuple $(a_0,a_1,\ldots,a_{n-2})$, there is an
    edge $(a_0,a_1,\ldots,a_{n-2},x)$ in $B^-_k(n-1)$, for some $x$, that is not in $B^-(S,n)$.
\item[vi)] if $n$ is even, for every vertex in $B^-_k(n-1)$ corresponding to a
    non-negatively-alternating right-sns $(n-1)$-tuple $(a_0,a_1,\ldots,a_{n-2})$, there is an
    edge $(y,a_0,a_1,\ldots,a_{n-2})$ in $B^-_k(n-1)$, for some $y$, that is not in $B^-(S,n)$.
\end{itemize}
\end{corollary}

\begin{proof}
The result follows immediately from Lemma~\ref{lemma:B-Sn_basic_properties}(ii) and
Corollary~\ref{corollary:nega-left-right-in-out}. \qed
\end{proof}

\subsection{Degree-parity Constraints on the Nega-sequence-subgraph}
\label{subsection:NOS-degree-parity-constraints}

We first give the following simple lemma.

\begin{lemma}  \label{lemma:ns-symmetric_not_symmetric}
Suppose $k>2$ and $n\geq2$.  If an $(n-1)$-tuple is both left-sns and negasymmetric then it is
uniform. Similarly, if an $(n-1)$-tuple is both right-sns and negasymmetric then it is uniform.
\end{lemma}

\begin{proof}
Let $\mathbf a=(a_0,\dots,a_{n-2})$ be an $(n-1)$-tuple which is both left-sns and negasymmetric.
Then $a_i=-a_{n-2-i}=-a_{n-3-i}$ for $i=0,1,\dots,n-2$. It is immediate that $a_j=a_{j+1}$ for
$j=0,1,\dots,n-3$ and thus $\mathbf a$ is uniform. The second claim follows by an analogous
argument. \qed
\end{proof}

The following lemma is key.

\begin{lemma}  \label{lemma:NOS_negasymmetric_vertex_even_degree}
Suppose $k>2$ and $n\geq2$.  Suppose $S$ is an $\mathcal{NOS}_k(n)$ and consider a vertex in
$B^-(S,n)$ with label $\mathbf{a}=(a_0,a_1,\ldots,a_{n-2})$, where $\mathbf{a}$ is
negasymmetric. Then $\mathbf{a}$ has even in-degree and even out-degree in $B^-(S,n)$.
\end{lemma}

\begin{proof}
Both $S$ and $-S^R$ correspond to an Eulerian circuit in $B^-(S,n)$, and these circuits are
edge-disjoint and cover all the edges of $B^-(S,n)$. If $\mathbf{a}$ is negasymmetric then both
circuits pass through this vertex equally many times. It follows that $\mathbf{a}$ has even
in-degree and even out-degree in $B^-(S,n)$. \qed
\end{proof}

We also need the following.

\begin{lemma}  \label{lemma:B-vertex-parity}
Suppose $k>2$ and $n\geq2$, and consider a vertex in $B^-_k(n-1)$ with label
$\mathbf{a}=(a_0,a_1,\ldots,a_{n-2})$.  Then in all the following cases this vertex has odd
in-degree and odd out-degree in $B^-_k(n-1)$:
\begin{itemize}
\item[i)] if $k$ is odd and $\mathbf{a}$ is negasymmetric but not uniform;
\item[ii)] if $n$ is odd, $k$ is even and $\mathbf{a}$ is either uniform or alternating, where
    $a_i\in\{0,k/2\}$ for every $i$;
\item[iii)] if $n$ and $k$ are both even and $\mathbf{a}$ is negatively-alternating.
\end{itemize}
\end{lemma}

\begin{proof}
\begin{itemize}
\item[i)] Since $\mathbf{a}$ is negasymmetric but not uniform, by
    Lemma~\ref{lemma:ns-symmetric_not_symmetric} is cannot be left-sns or right-sns.
     Thus, by Lemma~\ref{lemma:B-_vertex_degree_k-1} it has in-degree and out-degree $k$, and
     since $k$ is odd the result follows.
\item[ii)] Since $\mathbf{a}$ is either uniform or alternating, where $a_i\in\{0,k/2\}$ for
    every $i$, by Lemma~\ref{lemma:combining_lns-rns-symmetric}(ii) it is both left-sns and
    right-sns. Thus, by Lemma~\ref{lemma:B-_vertex_degree_k-1} it has in-degree and out-degree
    $k-1$, and since $k$ is even the result follows.
\item[iii)] Since $\mathbf{a}$ is negatively-alternating, by
    Lemma~\ref{lemma:combining_lns-rns-symmetric}(iii) it is both left-sns and right-sns. Thus,
    by Lemma~\ref{lemma:B-_vertex_degree_k-1} it has in-degree and out-degree $k-1$. Since $k$
    is even the result follows.

\end{itemize}
\qed
\end{proof}

Combining Lemmas~\ref{lemma:NOS_negasymmetric_vertex_even_degree} and \ref{lemma:B-vertex-parity}
immediately gives the following important result.

\begin{corollary}  \label{corollary:excluded-tuples-by-parity}
Suppose $k>2$ and $n\geq2$ and $S$ is an $\mathcal{NOS}_k(n)$. Then in all the following cases
there are edges $(a_0,a_1,\ldots,a_{n-2},x)$ and $(y,a_0,a_1,\ldots,a_{n-2})$ in $B^-_k(n-1)$
that are not in $B^-(S,n)$.
\begin{itemize}
\item[i)] if $k$ is odd, for every vertex in $B^-_k(n-1)$ corresponding to a negasymmetric
    non-uniform $(n-1)$-tuple $(a_0,a_1,\ldots,a_{n-2})$;
\item[ii)] if $n$ is odd and $k$ is even, for every vertex in $B^-_k(n-1)$ corresponding to a
    uniform or alternating negasymmetric $(n-1)$-tuple $(a_0,a_1,\ldots,a_{n-2})$, where
    $a_i\in\{0,k/2\}$ for every $i$;
\item[iii)] if $n$ and $k$ are both even, for every vertex in $B^-_k(n-1)$ corresponding to a
    negatively-alternating negasymmetric $(n-1)$-tuple $(a,-a,a,-a,\dots,a)$.
\end{itemize}
\end{corollary}

\subsection{Constraint Interactions}  \label{subsection:NOS-constraint-interactions}

The results above indicate in what circumstances certain edges in $B^-_k(n-1)$ cannot occur in
$B^-(S,n)$ if $S$ is an $\mathcal{NOS}_k(n)$.  In particular we have considered cases where one of
the incoming edges to a vertex corresponding to a negasymmetric or right-sns $(n-1)$-tuple cannot
occur in $B^-(S,n)$, and also where one of the outgoing edges from a vertex corresponding to a
negasymmetric or left-sns $(n-1)$-tuple cannot occur in $B^-(S,n)$. While we would like to add
together the numbers of eliminated edges for each case, we need to ensure we avoid `double
counting'.

More specifically, we need to consider when an edge can be outgoing from a negasymmetric or
left-sns $(n-1)$-tuple and also incoming to a negasymmetric or right-sns $(n-1)$-tuple.  This
motivates the following result.

\begin{lemma}  \label{lemma:NOS-constraint-interactions}
Suppose $k>2$.  Suppose $\mathbf{a}=(a_0,a_1,\dots,a_{n-2})$ and
$\mathbf{b}=(a_1,a_2,\dots,a_{n-1})$ are $k$-ary $(n-1)$-tuples.
\begin{itemize}
\item[i)] If $n\geq3$, and $\mathbf{a}$ and $\mathbf{b}$ are both negasymmetric then
    \begin{itemize}
    \item if $n$ is even then $\mathbf{a}$ and $\mathbf{b}$ are both uniform or
        alternating, where $a_i\in\{0,k/2\}$ for every $i$;
    \item if $n$ is odd then $\mathbf{a}$ and $\mathbf{b}$ are negatively-alternating.
    \end{itemize}

\item[ii)] If $n\geq5$, $\mathbf{a}$ is negasymmetric and $\mathbf{b}$ is right-sns then there
    exist $c_j$, $0\leq j\leq2$, such that $a_{3i+j}=c_j$, for every $i$ and $0\leq j\leq 2$.
    Moreover if $n\equiv0\pmod3$ then $c_2\in\{0,k/2\}$ and $c_0=-c_1$; if $n\equiv1\pmod3$
    then $c_1\in\{0,k/2\}$ and $c_0=-c_2$; and if $n\equiv2\pmod3$ then $c_0\in\{0,k/2\}$ and
    $c_1=-c_2$.

\item[iii)] If $n\geq5$, $\mathbf{a}$ is left-sns and $\mathbf{b}$ is negasymmetric then there
    exist $c_j$, $0\leq j\leq2$, such that $a_{3i+j}=c_j$, for every $i$ and $0\leq j\leq 2$.
    Moreover if $n\equiv0\pmod3$ then $c_0\in\{0,k/2\}$ and $c_1=-c_2$; if $n\equiv1\pmod3$
    then $c_2\in\{0,k/2\}$ and $c_0=-c_1$; and if $n\equiv2\pmod3$ then $c_1\in\{0,k/2\}$ and
    $c_0=-c_2$.

\item[iv)] If $n\geq5$, $\mathbf{a}$ is left-sns and $\mathbf{b}$ is right-sns then there exist
    $c_j$, $0\leq j\leq3$, such that $a_{4i+j}=c_j$, for every $i$ and $0\leq j\leq 3$.
    Moreover if $n\equiv0\pmod4$ then $c_0=-c_1$ and $c_2=-c_3$; if $n\equiv1\pmod4$ then
    $c_0=-c_2$, $c_1\in\{0,k/2\}$, and $c_3\in\{0,k/2\}$; if $n\equiv2\pmod4$ then $c_0=-c_3$
    and $c_1=-c_2$; and if $n\equiv3\pmod4$ then $c_1=-c_3$, $c_0\in\{0,k/2\}$, and
    $c_2\in\{0,k/2\}$.
\end{itemize}
\end{lemma}

\begin{proof}
\begin{itemize}
\item[i)] By negasymmetry of $\mathbf{a}$ and $\mathbf{b}$, respectively, we have
    $a_i=-a_{n-2-i}$ and $a_{i+1}=-a_{n-1-i}$ for every $i$, $0\leq i\leq n-2$.  Hence there
    exist $c_0$ and $c_1$ such that $a_{2i}=c_0$ and $a_{2i+1}=c_1$ for every $i$.
     \newline If $n$ is even then from the first equation we have
    $a_{n/2-1}=-a_{n/2-1}$, i.e.\ $a_{n/2-1}\in\{0,k/2\}$, and from the second equation we have
    $a_{n/2}=-a_{n/2}$, i.e.\ $a_{n/2}\in\{0,k/2\}$; the result follows.
     \newline If $n$ is odd then from the first equation we have $a_{(n-1)/2-1}=-a_{(n-1)/2-2}$,
    i.e.\ $c_0=-c_1$, and the result follows.

\item[ii)] By negasymmetry of $\mathbf{a}$, $a_i=-a_{n-2-i}$ for every $i$ ($0\leq i\leq n-2$),
    and by right-semi-negasymmetry of $\mathbf{b}$, $a_j=-a_{n+1-j}$ for every $j$ ($2\leq
    j\leq n-1$). Hence $a_i=a_{i+3}$, $0\leq i\leq n-3$. Now, since $n\geq5$, $a_{3i+j}=c_j$,
    for every $i$ and $0\leq j\leq 2$, for some $c_j$.
     \newline If $n\equiv0\pmod3$ then, by negasymmetry, $c_0=a_0=-a_{n-2}=-c_1$, i.e.\ $c_0=-c_1$,
    and by right-semi-negasymmetry we have $c_2=a_2=-a_{n-1}=-c_2$, i.e.\ $c_2\in\{0,k/2\}$.
     \newline If $n\equiv1\pmod3$ then, by negasymmetry, $c_0=a_0=-a_{n-2}=-c_2$ and so $c_0=-c_2$, and
     again by negasymmetry $c_1=a_1=-a_{n-3}=-c_1$, i.e.\ $c_1\in\{0,k/2\}$.
     \newline If $n\equiv2\pmod3$ then, by negasymmetry, $c_1=a_1=-a_{n-3}=-c_2$, i.e.\
     $c_1=-c_2$, and again by negasymmetry $c_0=a_0=-a_{n-2}=-c_0$, i.e.\ $c_0\in\{0,k/2\}$.

\item[iii)] By left-semi-negasymmetry of $\mathbf{a}$, $a_i=-a_{n-3-i}$ for every $i$ ($0\leq
    i\leq n-3$), and by negasymmetry of $\mathbf{b}$, $a_j=-a_{n-j}$ for every $j$ ($1\leq
    j\leq n-1$). Using exactly the same argument as for (ii), it follows that $a_{3i+j}=c_j$,
    for every $i$ and $0\leq j\leq 2$, for some $c_j$.
     \newline If $n\equiv0\pmod3$ then, by negasymmetry, $c_1=a_1=-a_{n-1}=-c_2$ and so $c_1=-c_2$,
    and by left-semi-negasymmetry $c_0=a_0=-a_{n-3}=-c_0$, i.e.\ $c_0\in\{0,k/2\}$.
     \newline If $n\equiv1\pmod3$ then, by negasymmetry, $c_1=a_1=-a_{n-1}=-c_0$ and so $c_1=-c_0$,
    and again by negasymmetry $c_2=a_2=-a_{n-2}=-c_2$, i.e.\ $c_2\in\{0,k/2\}$.
     \newline If $n\equiv2\pmod3$ then, by left-semi-negasymmetry, $c_0=a_0=-a_{n-3}=-c_2$ and so
    $c_0=-c_2$, and again by negasymmetry $c_1=a_1=-a_{n-1}=-c_1$, i.e.\ $c_1\in\{0,k/2\}$.

\item[iv)] By left-semi-negasymmetry of $\mathbf{a}$, $a_i=-a_{n-3-i}$ for every $i$ ($0\leq
    i\leq n-3$), and by right-semi-negasymmetry of $\mathbf{b}$, $a_{j}=-a_{n+1-j}$ for every
    $j$ ($2\leq j\leq n-1$).  Hence $a_i=a_{i+4}$, $0\leq i\leq n-3$.  Thus, since $n\geq5$,
    $a_{4i+j}=c_j$, for every $i$ and $0\leq j\leq 3$, for some $c_j$.
     \newline If $n\equiv0\pmod4$ then, by left-semi-negasymmetry, $c_0=a_0=-a_{n-3}=-c_1$, and so
    $c_0=-c_1$, and by right-semi-negasymmetry, $c_2=a_2=-a_{n-1}=-c_3$, and so $c_2=-c_3$.
     \newline If $n\equiv1\pmod4$ then, by left-semi-negasymmetry, $c_0=a_0=-a_{n-3}=-c_2$, by
    left-semi-negasymmetry, $c_1=a_1=-a_{n-4}=-c_1$, i.e.\ $c_1\in\{0,k/2\}$, and by
    right-semi-negasymmetry, $c_3=a_3=-a_{n-2}=-c_3$, and so $c_3\in\{0,k/2\}$.
     \newline If $n\equiv2\pmod4$ then, by left-semi-negasymmetry, $c_0=a_0=-a_{n-3}=-c_3$, and so
    $c_0=-c_3$, and by right-semi-negasymmetry, $c_2=a_2=-a_{n-1}=-c_1$, and so $c_2=-c_1$.
     \newline If $n\equiv3\pmod4$ then, by left-semi-negasymmetry, $c_1=a_1=-a_{n-4}=-c_3$, by
    left-semi-negasymmetry, $c_0=a_0=-a_{n-3}=-c_0$, i.e.\ $c_0\in\{0,k/2\}$, and by
    right-semi-negasymmetry, $c_2=a_2=-a_{n-1}=-c_2$, and so $c_2\in\{0,k/2\}$.

\end{itemize}
\qed
\end{proof}

\subsection{Developing the Bound}

Before establishing our new period bound, we first outline our strategy, enabling the proof of the
bound to be described more simply.

In Sections~\ref{subsection:NOS_in-out-constraints} and
\ref{subsection:NOS-degree-parity-constraints} we described two ways of showing that certain edges
in $B^-_k(n-1)$ cannot occur in $B^-(S,n)$ when $S$ is a $\mathcal{NOS}_k(n)$.  In particular we
showed that in two cases incoming edges to certain categories of vertex cannot occur, and also in
two cases that outgoing edges from certain categories of vertex cannot occur.

This suggests a straightforward strategy for bounding the period of an $\mathcal{NOS}_k(n)$,
namely that the period is at most half the maximum cardinality of edges in $B^-(S,n)$ (by
Lemma~\ref{lemma:B-Sn_basic_properties}(i)).  In turn, the number $|B^-(S,n)|$ of edges in
$B^-(S,n)$ is bounded above by the number of edges in $B^-_k(n-1)$, i.e.\ $N_k(n)$, as
specified in Lemma~\ref{lemma:number_non-negasymmetric_tuples}, less the number of edges in
$B^-_k(n-1)$ that cannot occur in $B^-(S,n)$, as specified in
Corollaries~\ref{corollary:NOS_excluded_tuples_by_unequal_degree} and
\ref{corollary:excluded-tuples-by-parity}.

This strategy is complicated by the fact that, as noted in
Section~\ref{subsection:NOS-constraint-interactions}, there is a danger of `double counting'
certain excluded edges.  For example,
Corollary~\ref{corollary:NOS_excluded_tuples_by_unequal_degree} asserts that certain outgoing edges
from a left-sns $(n-1)$-tuple cannot occur, and Corollary~\ref{corollary:excluded-tuples-by-parity}
asserts that certain edges incoming to a uniform or negasymmetric $(n-1)$ tuple cannot occur. These
two sets of excluded edges may overlap, and hence we need to take this into account;
Lemma~\ref{lemma:NOS-constraint-interactions} is of key importance in this respect.

The following notation is intended to simplify the arguments.  Let $U_{\text{in}}$ and
$U_{\text{out}}$ be the sets of incoming and outgoing edges excluded by
Corollary~\ref{corollary:NOS_excluded_tuples_by_unequal_degree}, and $P_{\text{in}}$ and
$P_{\text{out}}$ be the sets of incoming and outgoing edges excluded by
Corollary~\ref{corollary:excluded-tuples-by-parity}.

The above discussion leads to the following key lemma.

\begin{lemma}  \label{lemma:NOS-bounds-high-level}
If $k>2$ and $S$ is an $\mathcal{NOS}_k(n)$ then:
\[ |B^-(S,n)|\leq \begin{cases}
N_k(2) - |P_{\text{out}}|,
 & \text{if $n=2$} \\
N_k(3) - |P_{\text{out}}|-|P_{\text{in}}| + |P_{\text{out}}\cap P_{\text{in}}|,
 & \text{if $n=3$} \\
N_k(4) - |U_{\text{out}}|-|P_{\text{out}}|,
 & \text{if $n=4$} \\
N_k(n) - |U_{\text{out}}|-|U_{\text{in}}|-|P_{\text{out}}|-|P_{\text{in}}| \\
~~~~+~|U_{\text{out}}\cap P_{\text{in}}|+|P_{\text{out}}\cap U_{\text{in}}| \\
~~~~+~|U_{\text{out}}\cap U_{\text{in}}|+|P_{\text{out}}\cap P_{\text{in}}|,
 & \text{if $n>4$.}
\end{cases} \]
where $|X\cap Y|$ denotes the maximum possible cardinality for such a set, and $N_k(n)$ denotes the
number of non-negasymmetric $k$-ary $n$-tuples (see
Lemma~\ref{lemma:number_non-negasymmetric_tuples}).
\end{lemma}

\begin{proof}
The argument for $n=2$ is immediate. A similar comment applies when $n=3$.  The $n=4$ and $n \ge 5$
cases follow from observing that $U_{\text{in}}\cap P_{\text{in}}=U_{\text{out}}\cap
P_{\text{out}}=\emptyset$ because a negasymmetric $(n-1)$-tuple can neither be non-uniform left-sns
nor non-uniform right-sns, from Lemma~\ref{lemma:ns-symmetric_not_symmetric}. \qed
\end{proof}

\begin{remark}
The reason to restrict the sets considered when $n\leq 4$ is because certain sets are empty or
might be equal for small $n$.  Moreover
Corollary~\ref{corollary:NOS_excluded_tuples_by_unequal_degree} only applies for $n\geq4$.
\end{remark}

\subsection{The Bound}

We can now give the main result.

\begin{theorem}  \label{theorem:new_NOS_bounds}
Suppose $k>2$ and $S=(s_i)$ is an $\mathcal{NOS}_k(n)$ of period $m$.
\begin{itemize}
\item[i)] If $n=2$ then:
\[ m \leq \begin{cases}
\frac{k^2-k}{2}, & \text{if $k$ is odd} \\
\frac{k^2-k-2}{2}, & \text{if $k$ is even}
\end{cases} \]
\item[ii)] If $n=3$ then:
\[ m \leq \begin{cases}
\frac{k^3-3k+2}{2}, & \text{if $k$ is odd} \\
\frac{k^3-2k-4}{2}, & \text{if $k$ is even}
\end{cases} \]
\item[iii)] If $n=4$ then:
\[ m \leq \begin{cases}
\frac{k^4-2k^2+1}{2}, & \text{if $k$ is odd} \\
\frac{k^4-2k^2+k-2}{2}, & \text{if $k$ is even}
\end{cases}\]
\item[iv)] If $n>4$ then:
\[ m \leq \begin{cases}
\frac{k^n-5k^{(n-1)/2}+4k}{2}, & \text{if $n$ and $k$ are odd} \\
\frac{k^n-6k^{(n-1)/2}+4k}{2}, & \text{if $n$ is odd and $k$ is even} \\
\frac{k^n-3k^{n/2}-2k^{(n-2)/2}+k^2+3k}{2}, & \text{if $n$ is even and $k$ is odd} \\
\frac{k^n-3k^{n/2}+k^2+k-4}{2}, & \text{if $n$ and $k$ are even.}
\end{cases}\]
\end{itemize}
\end{theorem}

\begin{proof}
The proof builds on, and uses the notation of, Lemma~\ref{lemma:NOS-bounds-high-level}.
\begin{itemize}

\item[i)] {\bf $n=2$}.  If $k$ is odd then $P_{\text{out}}=\emptyset$ by
    Corollary~\ref{corollary:excluded-tuples-by-parity}(i), since there are no negasymmetric
    non-uniform 1-tuples.  If $k$ is even then, by
    Corollary~\ref{corollary:excluded-tuples-by-parity}(iii), $|P_{\text{out}}|=2$, since there
    are 2 negatively-alternating negasymmetric 1-tuples by
    Lemma~\ref{lemma:NOS_tuple_numbers}(vi). The result follows from
    Lemma~\ref{lemma:NOS-bounds-high-level}.

\item[ii)] {\bf $n=3$}.  If $k$ is odd then suppose $S$ is a $\mathcal{NOS}_k(3)$.  There are
    $k$ negasymmetric 3-tuples, namely $(0,0,0)$ and $(s,0,k-s)$, $1\leq s\leq k-1$.  Since the
    in-degree of every vertex in $B^-(S,3)$ is the same as its out-degree, at least one of the
    3-tuples corresponding to an incoming edge to the $k-1$ vertices $(s,0)$, $1\leq s\leq
    k-1$, cannot occur in $B^-(S,3)$. Also, from
    Corollary~\ref{corollary:excluded-tuples-by-parity}(i), at least one of the 3-tuples
    corresponding to an incoming edge to the $k-1$ vertices $(s,k-s)$, $1\leq s\leq k-1$,
    cannot occur in $B^-(S,3)$.  Thus at least $2(k-1)$ edges cannot occur.  Put together with
    the $k$ negasymmetric $k$-ary 3-tuples, this means that at least $3k-2$ edges cannot occur
    in $B^-(S,3)$, and the result follows.

    If $k$ is even then, by Corollary~\ref{corollary:excluded-tuples-by-parity}(ii),
    $P_{\text{out}}$ and $P_{\text{in}}$ only contain edges out-going/in-going from/to
    negasymmetric $2$-tuples that are uniform or alternating and have both entries equal to
    either $0$ or $k/2$.  Since $2$ is even, there are no negasymmetric $2$-tuples that are
    alternating and have every entry equal to $0$ or $k/2$, i.e.\ $P_{\text{out}}$ and
    $P_{\text{in}}$ only contain edges out-going/in-going from/to the two uniform $2$-tuples
    $0^2$ and $(k/2)^{2}$, i.e.\ $|P_{\text{out}}|=|P_{\text{in}}|=2$. Any edge in
    $P_{\text{out}}\cap P_{\text{in}}$ must equal $0^{3}$ or $(k/2)^{3}$.  But these are both
    negasymmetric and hence are not in $B^-_k(2)$; thus $|P_{\text{out}}\cap P_{\text{in}}|=0$.

    The result follows from Lemma~\ref{lemma:NOS-bounds-high-level}.

\item[iii)] {\bf $n=4$}. By Corollary~\ref{corollary:NOS_excluded_tuples_by_unequal_degree}(v),
    $|U_{\text{out}}|$ equals the number of non-negatively-alternating left-sns $3$-tuples,
    i.e.\ $k^2-k$ by Lemma~\ref{lemma:NOS_tuple_numbers}(x).  If $k$ is odd then, by
    Corollary~\ref{corollary:excluded-tuples-by-parity}(i), $|P_{\text{out}}|$ equals the
    number of negasymmetric non-uniform 3-tuples, i.e.\ $k-1$ by
    Lemma~\ref{lemma:NOS_tuple_numbers}(i),(v). If $k$ is even then, by
    Corollary~\ref{corollary:excluded-tuples-by-parity}(iii), $|P_{\text{out}}|$ equals the
    number of negatively-alternating negasymmetric 3-tuples, i.e.\ 2 by
    Lemma~\ref{lemma:NOS_tuple_numbers}(vi). The result follows from
    Lemma~\ref{lemma:NOS-bounds-high-level}.

\item[iv)a)] {\bf $n>4$; $n$ and $k$ odd}.  By
    Corollary~\ref{corollary:NOS_excluded_tuples_by_unequal_degree}(i),(ii),
    $|U_{\text{out}}|=|U_{\text{in}}|$ equals the number of non-uniform left-sns
    $(n-1)$-tuples, i.e.\ $k^{(n-1)/2}-1$ by Lemma~\ref{lemma:NOS_tuple_numbers}(ix).

    Also, by Lemma~\ref{lemma:NOS-constraint-interactions}(iv), $|U_{\text{out}}\cap
    U_{\text{in}}|=k-1$ since if $n\equiv 1\pmod 4$ there are $k$ choices for $c_0$, and one
    choice each for $c_1$ and $c_3$, giving $k$ possibilities, one of which is uniform (when
    $c_0=c_1=c_3=0$). An analogous argument applies if $n\equiv 3 \pmod 4$.

    By Corollary~\ref{corollary:excluded-tuples-by-parity} (i),
    $|P_{\text{out}}|=|P_{\text{in}}|$ equals the number of negasymmetric non-uniform
    $(n-1)$-tuples, i.e.\ $k^{(n-1)/2}-1$ by Lemma~\ref{lemma:NOS_tuple_numbers}(i),(v).

    By Lemma~\ref{lemma:NOS-constraint-interactions}(ii),(iii), $|U_{\text{out}}\cap
    P_{\text{in}}|=|P_{\text{out}}\cap U_{\text{in}}|=k-1$ (removing the case where the $c_i$
    are all equal).

    By Lemma~\ref{lemma:NOS-constraint-interactions}(i), $|P_{\text{out}}\cap P_{\text{in}}|$
    equals the number of negasymmetric non-uniform $(n-1)$-tuples that are also
    negatively-alternating, i.e.\ $k-1$ by Lemma~\ref{lemma:NOS_tuple_numbers}(iv),(vi).

    The result follows from Lemma~\ref{lemma:NOS-bounds-high-level}.

\item[iv)b)] {\bf $n>4$; $n$ odd and $k$ even}.  By
    Corollary~\ref{corollary:NOS_excluded_tuples_by_unequal_degree}(iii),(iv),
    $|U_{\text{out}}|=|U_{\text{in}}|$ equals the number of non-uniform non-alternating
    left-sns $(n-1)$-tuples, i.e.\ $2k^{(n-1)/2}-4$ by Lemma~\ref{lemma:NOS_tuple_numbers}(xi).

    Also, by Lemma~\ref{lemma:NOS-constraint-interactions}(iv), $|U_{\text{out}}\cap
    U_{\text{in}}|=4k-4$ since if $n\equiv 1\pmod 4$ there are $k$ choices for $c_0$, and two
    choices each for $c_1$ and $c_3$, giving $4k$ possibilities, two of which are uniform and
    two of which are alternating. An analogous argument applies if $n\equiv 3 \pmod 4$.

    By Corollary~\ref{corollary:excluded-tuples-by-parity}(ii),
    $|P_{\text{out}}|=|P_{\text{in}}|$ equals the number of uniform or alternating
    negasymmetric $(n-1)$-tuples in which every entry is either $0$ or $k/2$, i.e.\ $2$ by
    Lemma~\ref{lemma:NOS_tuple_numbers}(v),(vii), both of which are uniform.

    By Corollary~\ref{corollary:excluded-tuples-by-parity}(ii), $P_{\text{out}}$ and
    $P_{\text{in}}$ only contain edges out-going/in-going from/to uniform or alternating
    $(n-1)$-tuples, and by
    Corollary~\ref{corollary:NOS_excluded_tuples_by_unequal_degree}(iii),(iv), $U_{\text{out}}$
    and $U_{\text{in}}$ only contain edges that are out-going/in-going from/to non-uniform
    non-alternating $(n-1)$-tuples, and hence $|U_{\text{out}}\cap
    P_{\text{in}}|=|P_{\text{out}}\cap U_{\text{in}}|=0$.

    By Corollary~\ref{corollary:excluded-tuples-by-parity}(ii), $P_{\text{out}}$ and
    $P_{\text{in}}$ only contain edges out-going/in-going from/to negasymmetric $(n-1)$-tuples
    that are uniform or alternating and have every entry equal to either $0$ or $k/2$.  Since
    $n-1$ is even, there are no negasymmetric $(n-1)$-tuples that are alternating and have
    every entry equal to $0$ or $k/2$, i.e.\ $P_{\text{out}}$ and $P_{\text{in}}$ only contain
    edges out-going/in-going from/to the two uniform $(n-1)$-tuples $0^{n-1}$ and
    $(k/2)^{n-1}$. But, since $n\geq3$, any edge in $P_{\text{out}}\cap P_{\text{in}}$ must
    equal $0^{n}$ or $(k/2)^{n}$.  But these are both negasymmetric and hence are not in
    $B^-_k(n-1)$; thus $|P_{\text{out}}\cap P_{\text{in}}|=0$.

    The result follows from Lemma~\ref{lemma:NOS-bounds-high-level}.

\item[iv)c)] {\bf $n>4$; $n$ even}. By
    Corollary~\ref{corollary:NOS_excluded_tuples_by_unequal_degree}(v),(vi),
    $|U_{\text{out}}|=|U_{\text{in}}|$ equals the number of non-negatively-alternating left-sns
    $(n-1)$-tuples, i.e.\ $k^{n/2}-k$ by Lemma~\ref{lemma:NOS_tuple_numbers}(x).

    Additionally, by Lemma~\ref{lemma:NOS-constraint-interactions}(iv), $|U_{\text{out}}\cap
    U_{\text{in}}|=k(k-1)$, by choosing $c_0$ and $c_2$ to be distinct.

    {\bf If $k$ is odd} then, by Corollary~\ref{corollary:excluded-tuples-by-parity}(i),
    $|P_{\text{out}}|=|P_{\text{in}}|$ equals the number of negasymmetric non-uniform
    $(n-1)$-tuples, i.e.\ $k^{(n-2)/2}-1$ by Lemma~\ref{lemma:NOS_tuple_numbers}(i),(v).
     \newline By Lemma~\ref{lemma:NOS-constraint-interactions}(ii),(iii), $|U_{\text{out}}\cap
    P_{\text{in}}|=|P_{\text{out}}\cap U_{\text{in}}|=k-1$, by ensuring that $c_0$,
    $c_1$ and $c_2$ are not all equal to $0$.
     \newline By Lemma~\ref{lemma:NOS-constraint-interactions}(i), $|P_{\text{out}}\cap
    P_{\text{in}}|=0$, i.e.\ the number of non-negatively-alternating $(n-1)$-tuples in which
    every entry is either 0 or $k/2$ when $k$ is odd.

    {\bf If $k$ is even} then, by Corollary~\ref{corollary:excluded-tuples-by-parity}(iii),
    $P_{\text{out}}$ and $P_{\text{in}}$ only contain edges out-going from or in-going to
    negatively-alternating negasymmetric $(n-1)$-tuples.  Since $n-1$ is odd, there are only
    two such $(n-1)$-tuples, namely $0^{n-1}$ and $(k/2)^{n-1}$, and thus
    $|P_{\text{out}}|=|P_{\text{in}}|=2$, in accordance with
    Lemma~\ref{lemma:NOS_tuple_numbers}(vi).  Hence any edge ($n$-tuple) in $P_{\text{out}}\cap
    P_{\text{in}}$ must equal either $0^{n}$ or $(k/2)^{n}$, both of which are negasymmetric
    and hence are not in $B^-_k(n-1)$; thus $|P_{\text{out}}\cap P_{\text{in}}|=0$.
     \newline By Corollary~\ref{corollary:excluded-tuples-by-parity}(iii), $P_{\text{out}}$
    and $P_{\text{in}}$ respectively only contain edges out-going from or in-going to
    negatively-alternating $(n-1)$-tuples, and by
    Corollary~\ref{corollary:NOS_excluded_tuples_by_unequal_degree}(v),(vi), $U_{\text{out}}$
    and $U_{\text{in}}$ only contain edges that are out-going/in-going from/to
    non-negatively-alternating $(n-1)$-tuples, and hence $|U_{\text{out}}\cap
    P_{\text{in}}|=|P_{\text{out}}\cap U_{\text{in}}|=0$.

    The result follows from Lemma~\ref{lemma:NOS-bounds-high-level}.

\end{itemize}
\qed
\end{proof}

\subsection{A Table}

The bounds of Theorem~\ref{theorem:new_NOS_bounds} are tabulated for small $k$ and $n$ in
Table~\ref{table:new_NOS_periods_bounds}.  Note that the numbers given in brackets are computed
from the bounds in Theorem~\ref{result:NOS-bounds}, and are provided for comparison purposes.

\begin{table}[htb]
\centering \caption{Bounds --- new and (old) --- on the period of an $\mathcal{NOS}_k(n)$}
\label{table:new_NOS_periods_bounds}
\begin{footnotesize}
\begin{tabular}{crrrrrrr} \hline
$n$ & $k=3$  & $k=4$   & $k=5$   & $k=6$    & $k=7$     & $k=8$     & $k=9$       \\ \hline
2   & 3      & 5       & 10      & 14       & 21        & 27        & 36         \\
    & (3)    & (5)     & (10)    & (14)     & (21)      & (27)      & (36)       \\ \hline
3   & 10     & 26      & 56      & 100      & 162       & 246       & 352        \\
    & (11)   & (27)    & (58)    & (101)    & (165)     & (247)     & (356)      \\ \hline
4   & 32     & 113     & 288     & 614      & 1152      & 1987      & 3200       \\
    & (35)   & (119)   & (298)   & (629)    & (1173)    & (2015)    & (3236)     \\ \hline
5   & 105    & 472     & 1510    & 3792     & 8295      & 16208     & 29340      \\
    & (113)  & (495)   & (1538)  & (3851)   & (8355)    & (16319)   & (29444)     \\ \hline
6   & 324    & 1960    & 7620    & 23023    & 58296     & 130338    & 264600     \\
    & (347)  & (2015)  & (7738)  & (23219)  & (58629)   & (130815)  & (265316)   \\ \hline
7   & 1032   & 8008    & 38760   & 139332   & 410928    & 1047056   & 2389680    \\
    & (1067) & (8127)  & (38938) & (139751) & (411429)  & (1048063)  & (2390756) \\ \hline
8   & 3141   & 32392   & 194270  & 837883   & 2878491   & 8382498   & 21512844   \\
    & (3227) & (32639) & (194938)& (839159) & (2881029) & (8386559) & (21519716) \\ \hline
9   & 9645   & 130312  & 975010  & 5034972  & 20170815  & 67096592  & 193693860  \\
    & (9761) & (130815)& (975938)& (5037551)& (20174403)& (67104767)& (193703684)\\ \hline
\end{tabular}
\end{footnotesize}

\end{table}

Observe that since, for $k$ odd,  $r_{k,3,3k/2}=3k-2$\footnote{Observing that $r_{k,3,3k/2}$ is
the number of $k$-ary 3-tuples with pseudoweight exactly $3k/2$ then, since $3k$ is odd, a
3-tuple with pseudoweight exactly $3k/2$ must contain either one or three zeros. Thus the set
of 3-tuples with pseudoweight $3k/2$ is $\{(0,0,0)\}\cup\{(s,k-s,0), (s,0,k-s), (0,s,k-s):
1\leq s\leq k-1\}$.}, and from \cite[Lemma 3.9]{Mitchell25a} there exists a
$\mathcal{NOS}_k(n)$ with period $(k^n-r_{k,3,3k/2})/2$,
Theorem~\ref{theorem:new_NOS_bounds}(ii) implies that the sequences of \cite[Lemma
3.9]{Mitchell25a} have maximal period for $n=3$ and odd $k$.

\section{Concluding Remarks}  \label{section:conclusions}

Although for $n>2$ the new bounds are sharper than those previously known, there remains a gap
between the largest known period of a $\mathcal{NOS}_k(n)$ and the bound, even for small values of
$n$. The current state of knowledge for small $n$ and $k>2$ is summarised in
Table~\ref{table:largest_NOS_periods}, where the upper bound from
Theorem~\ref{theorem:new_NOS_bounds} is given in brackets beneath the largest known period of a
$\mathcal{NOS}_k(n)$. The sequences with the largest known period are derived from \cite[Theorem
4.9]{Mitchell26a}. The values in bold represent maximal values.

\begin{table}[htb]
\centering \caption{Largest known periods for an $\mathcal{NOS}_k(n)$ (and bounds)}
\label{table:largest_NOS_periods}

\begin{tabular}{crrrrrr} \hline
$n$ & $k=3$   & $k=4$    & $k=5$    & $k=6$    & $k=7$     & $k=8$     \\ \hline
2   & {\bf 3} & {\bf 5}  & {\bf 10} & {\bf 14} & {\bf 21}  & {\bf 27}  \\
    & (3)     & (5)      & (10)     & (14)     & (21)      & (27)      \\ \hline
3   & {\bf 10}& 22       & {\bf 56} & 92       & {\bf 162} & 234       \\
    & (10)    & (26)     & (56)     & (100)    & (162)     & (246)     \\ \hline
4   & 31      & 99       & 284      & 591      & 1146      & 1955      \\
    & (32)    & (113)    & (288)    & (614)    & (1152)    & (1987)    \\ \hline
5   & 101     & 436      & 1502     & 3712     & 8283      & 16068     \\
    & (105)   & (472)    & (1510)   & (3792)   & (8295)    & (16208)   \\ \hline
6   & 316     & 1870     & 7596     & 22808    & 58248     & 129946    \\
    & (324)   & (1960)   & (7620)   & (23023)  & (58296)   & (130338)  \\ \hline
7   & 1002    & 7750     & 38628    & 138462   & 410574    & 1044998   \\
    & (1032)  & (8008)   & (38760)  & (139332) & (410928)  & (1047056) \\ \hline
8   & 3068    & 31751    & 193816   & 835495   & 2876916   & 8376327   \\
    & (3141)  & (32392)  & (194270) & (837883) & (2878491) & (8382498) \\ \hline
\end{tabular}
\end{table}

While it was already known \cite[Note 3.12]{Mitchell25a} how to construct optimal period negative
orientable sequences for $n=2$, and we have established that the negative orientable sequences of
\cite[Lemma 3.9]{Mitchell25a} have optimal period for $n=3$ and $k$ odd, it should be clear that
determining the optimal period for a $\mathcal{NOS}_k(n)$ remains an open question for $n=3$ and
$k$ even, and all $n>3$ regardless of $k$. Addressing this question remains an area for future
research. Indeed, the smallest open cases (i.e.\ $n=3$, $k=4$ and $n=4$, $k=3$) appear readily
resolvable by computer search.


\providecommand{\bysame}{\leavevmode\hbox to3em{\hrulefill}\thinspace}
\providecommand{\MR}{\relax\ifhmode\unskip\space\fi MR }
\providecommand{\MRhref}[2]{%
  \href{http://www.ams.org/mathscinet-getitem?mr=#1}{#2}
} \providecommand{\href}[2]{#2}

\end{document}